\documentclass{article}
\usepackage{amsmath}
\usepackage{amssymb}
\usepackage{amsthm}
\usepackage{bm}
\usepackage{geometry}
\usepackage{color}

\newtheorem{definition}{Definition}[section]
\newtheorem{theorem}{Theorem}[section]
\newtheorem{proposition}{Proposition}[section]
\newtheorem{remark}{Remark}
\newtheorem{corollary}{Corollary}
\newtheorem{lemma}{Lemma}

\begin{document}

		\title{A new interpolation method for metric spaces based on bi-infinite sequences: The $R$-Method}
	\author{Robledo Mak's Miranda Sette}
	\date{\today}
	\maketitle

\begin{abstract}
	We introduce a new interpolation method for metric spaces, termed the $R$-method, based
	on bi-infinite linking sequences. Although the construction is inspired by the classical
	metric functional $J_M$, the resulting interpolated space is generated by a distinct
	object that behaves as a multiscale energy functional. This functional measures the
	minimal discrete action required to connect two points through $\mathbb{Z}$-indexed
	sequences, leading to a new intrinsic metric on $X_0 \cap X_1$.
	
	The associated interpolated space is obtained as the relative completion of this metric
	inside $X_0 \cup X_1$ and is genuinely different from those produced by the $J_M$- and
	$K_M$-methods. A fundamental structural property of the $R$-method is that the resulting
	space embeds continuously into the corresponding $K_M$-interpolated space, situating the
	construction naturally within the existing theory of metric interpolation.
	
	When the method is restricted to a normed setting, the $R$-method induces a genuine
	interpolation functor. In this framework, it preserves the Lipschitz property of
	operators with closed graphs, even in the absence of linearity, thereby extending the
	classical scope of interpolation theory, which is traditionally confined to linear
	continuous operators. As a consequence, standard compactness properties are also
	preserved under mild assumptions.
	
	The $R$-method thus provides a new interpolation framework whose foundations rely
	exclusively on intrinsic metric properties and the summability of discrete orbits,
	bridging metric interpolation, nonlinear analysis, and classical interpolation theory.
\end{abstract}


\paragraph{Classification codes:} 46B70, 54E35, 47H09.

\section*{Introduction}

The classical theory of interpolation, initiated in the 1960s through the foundational
work of Lions and Peetre \cite{LionsPeetre1964} and later systematized in texts such as
\cite{BerghLofstrom1976,BrudnyiKruglyak1991}, is intrinsically rooted in linear and normed
structures. Interpolation functors were designed to construct intermediate Banach spaces
between two given endpoints, providing powerful tools for the analysis of linear
operators, partial differential equations, harmonic analysis, and geometric functional
analysis. For decades, the theory evolved almost exclusively within linear frameworks,
with its principal constructions relying on convexity, homogeneity, and duality.

Only recently has the need for interpolation tools beyond the linear setting become
apparent. Several modern areas, including optimal transport, metric geometry, gradient
flows in Wasserstein spaces, and nonlinear analysis on general metric measure spaces, are
naturally formulated in purely metric terms and lack vector space structure. Nevertheless,
many of these contexts exhibit intermediate regularity phenomena that are conceptually
analogous to interpolation. This has motivated the search for interpolation frameworks
that operate directly at the level of metric spaces, without recourse to linear or
normed structures. Despite progress in nonlinear functional analysis and metric geometry,
a systematic and flexible interpolation theory for arbitrary metric spaces has remained
largely undeveloped.

In \cite{Robledo2024}, we initiated such a program by introducing two metric-space
analogues of Peetre’s classical $K$- and $J$-methods, denoted respectively by the
$K_M$- and $J_M$-methods. These constructions reinterpret the classical ideas entirely in
metric terms. Given a compatible pair of metric spaces $(X_0,X_1)_X$, the $K_M$-method is
based on admissible finite linking sequences in $X_0\cup X_1$, penalizing each step through
a functional of the form
\[
K_M(t;x,y)=\inf \sum h_t(x_k,x_{k+1}),
\]
which mirrors the behavior of the classical $K$-functional. In contrast, the $J_M$-method
acts on the intersection $X_0\cap X_1$ and is governed by the metric
\[
J_M(t;x,y)=\max\{d_0(x,y),\,t\,d_1(x,y)\},
\]
the natural metric counterpart of Peetre’s $J$-functional.

Both methods generate interpolated metric spaces via relative completion: one first
endows $X_0\cap X_1$ with a suitable multiscale metric functional and then completes it
inside the ambient space $(X_0\cup X_1,K_M(1))$. These constructions preserve Lipschitz
continuity of operators, exhibit functorial behavior analogous to that of classical
interpolation methods, and apply naturally in purely metric settings, including
Wasserstein spaces \cite{AmbrosioGigliSavare2008,Villani2003,Villani2009}. Their introduction
represents one of the first systematic attempts to extend interpolation theory beyond the
linear world.

The present article introduces a third interpolation method, which is structurally
distinct from both the $K_M$- and $J_M$-methods. Although the construction still involves
the classical functional $J_M$, the underlying mechanism is fundamentally different.
Rather than relying on finite admissible chains or single-step comparisons, the new
method is built upon \emph{bi-infinite linking sequences}
\[
(S_k)_{k\in\mathbb{Z}}\subset X_0\cap X_1,
\]
which simultaneously approximate a point $x$ as $k\to+\infty$ and a point $y$ as
$k\to-\infty$. Along such sequences, we evaluate a multiscale functional of the form
\[
\Gamma_{\theta,q}\bigl(J_M(2^k;S_k,S_{k+1})\bigr),
\]
measuring the minimal discrete action required to connect $x$ and $y$ across all dyadic
scales. Taking the infimum over all admissible approximating sequences yields a pre-metric
$p_{\theta,q}$; symmetrization and path-length completion then produce a genuine metric
$\delta_{\theta,q}$ on $X_0\cap X_1$.

The resulting interpolated space is defined as the relative completion of
$(X_0\cap X_1,\delta_{\theta,q})$ inside $(X_0\cup X_1,K_M(1))$. This construction yields a
new metric interpolation space, denoted by $R_{\theta,q}^{\overrightarrow{X}}$, whose
geometry differs essentially from that of the previously studied $K_M$- and
$J_M$-interpolated spaces. The use of bi-infinite approximating sequences introduces a
variational and multiscale structure, capturing information distributed across all
dyadic scales. In this respect, the method is closer in spirit to energetic approaches
common in modern metric analysis, while remaining firmly rooted in interpolation theory.

From a broader perspective, the development of interpolation theory for metric spaces
remains in its early stages, and the availability of multiple mechanisms is essential.
Different analytic and geometric problems may require different ways of measuring
intermediate regularity or oscillation. The $R$-method enriches the existing toolkit by
providing a sequence-based, variational, and genuinely multiscale interpolation
mechanism, which may be advantageous in nonlinear settings where convergence of sequences
plays a central role, such as gradient flows, geometric evolution problems, and
approximation schemes in metric geometry.

The purpose of this article is to present the full construction of the $R$-method,
establish its fundamental metric properties, and show that it is compatible with the
interpolation of Lipschitz and compact operators. While the method is intrinsically
metric, we also show that, when restricted to normed spaces, it induces a genuine
interpolation functor, extending the classical scope of interpolation beyond linear
operators. In doing so, the article lays the foundations for a new and distinct
interpolation mechanism within the emerging theory of metric-space interpolation.

\section{Preliminaries}
All results and definitions in this section are found in the work [see \cite{Robledo2024}]. Thus, we decided to omit the proofs of basic lemmas of the theory that fall outside the scope of this text. The interested reader can verify the proofs in the reference above.

\subsection{Relative Completion of Metric Spaces}
\begin{definition} Let $(A, d_A)$ and $(B, d_B)$ be two metric spaces such that $A \subseteq B$ and
	there exists $c >0$ such that $$d_B(x, y) \leq c \cdot d_A(x, y), \quad x, y \in A.$$ We define the \textbf{relative completion of $(A, d_A)$ in $(B, d_B)$}, denoted by $A_B$, as the set of all the $x \in B$ such that there exists a Cauchy sequence $(x_n)$ in $(A, d_A)$ satisfying $d_B(x_n, x) \to 0$.
	
	A sequence $(x_n)$ as above is called an \textbf{approximating sequence of $x$ in $(A, d_A)$}.

\end{definition}
There are some intrinsic properties of the relative completion of metric spaces that we list below:
\begin{enumerate} \item This completion can be metrized with a metric $d_{A, B}$ that
	satisfies $$d_B(x, y) \leq c \cdot d_{A, B}(x, y) \leq d_A(x, y), \quad x, y \in A$$ and $$ d_B(x, y) \leq c \cdot d_{A, B}(x, y), \quad x, y \in A_B.$$
	Thus, $A \subseteq A_B \subseteq B$ with continuous inclusions.
	\item The set $A$ is dense in $(A_B, d_{A, B})$.
	\item If $(B, d_B)$ is complete, then $(A_B, d_{A, B})$ is also complete.

	\item The set $A$ is dense in $(A_B, d_{A, B})$.   \item The metrics $d_{A, B}$ and $d_A$ coincide on $A$.  \end{enumerate}
\subsection{Pairs of Compatible Metric Spaces}
\begin{definition}\label{def 13} Let $X_0$ and $X_1$ be two sets such that $X_0 \cap X_1 \neq \emptyset$ and
	both are subsets of a set $X$. Furthermore, assume that $d_0, d_1$ and $d_{X}$ are metrics on the sets $X_0$, $X_1$ and $X$, respectively, such that for convenient positive constants $C_0$ and $C_1$ we have	
	\begin{equation} d_{X}(x, y) \leq C_i \cdot  d_i(x, y), \quad \forall \, x, y \in X_i, \quad i \in \{0,1\}. \end{equation}
	We call this object a \textbf{pair of compatible metric spaces} and denote it by $\overrightarrow{X}$ or by $(X_0, X_1)_X$.
\end{definition}
Let $(X_0, X_1)_X$ and $t>0$. For $(x, y) \in X_0^2 \cup X_1^2$, define the functional \textit{hop}, denoted by $h_t$ which will be given by
\[ h_t(x, y) = \left\{ \begin{array}
	[c]{l} \min\{d_0(x, y), t \cdot d_1(x, y)\}, \text{ if } x, y \in X_0 \cap X_1; \\ d_0(x, y), \text{ if}  \,\, \{x, y\} \subseteq X_0 \setminus X_1; \\ t \cdot  d_1(x, y), \text{ if}  \,\, \{x, y\} \subseteq X_1 \setminus X_0.		
\end{array} \right. \]
\begin{definition} Given $x, y \in X_0 \cup X_1$, with $X_0 \cap X_1 \neq \emptyset$, \textbf{an admissible linking sequence from $x$ to $y$ in $X_0 \cup X_1$} is any finite sequence of points $x_0, x_1, \ldots, x_n\in X_0 \cup X_1$ that satisfies:
	\begin{itemize} \item $x_0=x$ and $x_n=y$; \item For each $k \in \{0, 1, \ldots, n-1\}$, $(x_k, x_{k+1}) \in X_0^2 \
		cup X_1^2$ holds. \end{itemize} Let us denote an admissible linking sequence $(x, x_1, x_2, \ldots, x_{n-1}, y)$ from $x$ to $y$ in $X_0 \cup X_1$ by $(x_n^{x, y})_{adm}$.
	When all $n$ points of the sequence belong to the same set $Z$, we will say that it is a \textbf{linking sequence from $x$ to $y$ in $Z$} and denote it only by $ (x_n^{x,y})$. The set of all linking sequences from $x$ to $y$ in $Z$ is denoted by $ [x, y]_Z$.
\end{definition}
\begin{remark}Note that every admissible linking sequence is a linking sequence in the union $X_0 \cup X_1$, but not every linking sequence in $X_0 \cup X_1$ is admissible.

\end{remark}
\subsection{The $K_M$ Functional and its properties}
\begin{definition}[The $K_M$-Metric Functional] Let $(X_0, X_1)_X$ be a pair of compatible metric spaces and $t>0$. Given $x, y \in X_0 \cup X_1$, we define the \textbf{$K_M$-metric functional} by
	\begin{equation} K_M(t; x, y):= \inf \, \sum_{k=0}^{n-1}h_t(x_k, x_{k+1}),
	\end{equation} where the infimum is taken over all admissible linking sequences from $x$ to $y$
	in the set $X_0 \cup X_1$.
\end{definition}
From now on, unless otherwise stated and when there is no possibility of confusion due to notation, $(X_0, d_0)$ and $(X_1, d_1)$ will be considered a pair of compatible metric spaces with respect to the space $(X, d_X)$, that is, they will represent $(X_0, X_1)_X$.

\begin{lemma}\label{prop 14} If $t \in (0, \infty)$, then the function $K_M(t; \cdot \, , \cdot \, )$ is a metric on $X_0 \cup X_1$.
\end{lemma}
\begin{lemma} \label{desigualdade dos K} If $a, b \in (0, \infty)$ and $x, y \in X_0 \cup X_1$, then
	\begin{equation} K_M(a; x, y) \leq \max \biggr\{1, \frac{a}{b} \biggr \} \cdot K_M(b; x, y)
	\end{equation}
\end{lemma}
\subsection{The $J_M$ Functional}
\begin{lemma}\label{prop 15} If $x, y \in X_0 \cap X_1$ and $t>0$, defining $$J_M(t; x, y):= \max\{ d_0(x, y), t \cdot  d_1(x, y)\},$$ then, $J_M(t; \,\, \cdot \,\,, \,\, \cdot \,\,)$ is a metric on $X_0 \cap X_1$.

\end{lemma}
\begin{lemma}\label{lema 10} If $a, b \in (0, \infty)$ and $x, y \in X_0 \cap X_1$, then \begin{equation}
		J_M(a; x, y) \leq \max\biggr\{1, \frac{a}{b}\biggr\}\cdot  J_M(b; x, y), \quad x, y \in X_0 \cap X_1;
\end{equation} \end{lemma}
\begin{lemma} \label{lema 12} If $x, y \in X_0 \cap X_1$ and $a, b \in (0, \infty)$, then \begin{equation}
		K_M(a; x, y) \leq \min \biggr\{1, \frac{a}{b} \biggr\}\cdot  J_M(b; x, y). \end{equation}
\end{lemma}
\section{The $K_M$-Metric Interpolation Space}

In this section we will define and highlight the main properties of the interpolated metric space.

\begin{definition} Let $\theta \in (0, 1)$ and $q \in [1, \infty]$. Given a sequence $(x_k)_{k \in \
		\mathbb{Z}}$ of real numbers, we define the functional $\Gamma_{\theta, q}$ by \begin{equation} \Gamma_{\theta, q}((x_k)_{k \in \mathbb{Z}}):= \begin{cases}
			\biggr[ \sum_{k \in \mathbb{Z}} [2^{-k \theta} |x_k|]^q \biggr]^{1/q},& q< \infty\\				
			\sup_{k \in \mathbb{Z}} 2^{-k \theta}|x_k|, & q= \infty. \end{cases}
\end{equation} \end{definition}
\begin{lemma} The functional \begin{equation} \beta_{\theta, q}(x, y):= \Gamma_{\theta, q}((K_M(2^k; x, y)_{k \in \mathbb{Z}})
	\end{equation} is a metric on the set $X_0 \cap X_1$.
\end{lemma}
We will refer, for simplicity, to the metrics $J_M(t; \cdot, \cdot)$ and $K_M(t; \cdot , \cdot)$ as $J_M(t)$ and $K_M(t)$, respectively.
\begin{definition}[The $K_M$-interpolated metric space] The \textbf{interpolated metric space by the classical $K_M$ method} is denoted by $\overrightarrow{X}_{\theta, q}^{K_M}$ and is the relative completion of the metric space $$(X_0 \cap X_1, \beta_{\theta, q})$$ in the metric space $(X_0 \cup X_1, K_M(1))$. The metric of this space is denoted by $\mathcal{D}_{\theta, q}$ and equals $\mathcal{D}_{\theta, q}(x, y) = \beta_{\theta, q}(x, y)$, for all $x, y \in X_0 \cap X_1$.

\end{definition}
\begin{lemma} Given $\theta \in (0, 1)$ and $q \geq 1$, we have \begin{equation} \beta_{\theta, q} (x, y) \leq M_{\theta, q} \cdot J_M(1; x, y), \quad x, y \in X_0 \cap
		X_1. \end{equation} And furthermore, \begin{equation}
		M_{\theta, q} \cdot d_X(x, y) \leq \mathcal{D}_{\theta, q}(x, y), \quad x, y \in \overrightarrow{X}_{\theta, q}^{K_M}.
	\end{equation}
	where \begin{equation*}{\label{mtetaq}}
		M_{\theta, q} = \displaystyle  \Gamma_{\theta, q} (\min \{1, 2^k\}_{k \in \mathbb{Z}}) >0. \end{equation*}

\end{lemma}
\section{The $J_M^{\prime}$-Interpolated Metric Space}
\begin{definition}[Bi-infinite Sequences] A sequence is said to be \textbf{bi-infinite} when it is indexed in the set of integers. A bi-infinite sequence $(z_k)_{k \in \mathbb{Z}}$
	is said to converge to the point $z$ in a metric space $(M, d)$ if $$\lim_{n \to \infty} d(z_n, z)=0 \quad \text{and} \quad \lim_{n \to - \infty} d(z_n, z)=0.$$ \end{definition}
We know that $X_0 \cap X_1 \subseteq X_0 \cup X_1$ and that $K_M(1)$ is a metric on $X_0 \cup X_1$. We will metrize the set $X_0 \cap X_1$ with a metric satisfying the relative completion condition and then complete this space equipped with this metric, with respect to the space $(X_0 \cup X_1, K_M(1))$.

Given $x, y \in X_0 \cap X_1$ and $(x_n)$ and $(y_n)$ approximating sequences of $x$ and $y$ in $ (X_0 \cap X_1, J_M(1))$, respectively, let us define a new bi-infinite sequence, denoted by $(S_k^{x, y})_{k \in \mathbb{Z}}$, which depends on these approximating sequences and whose general term is given by: $$S_k^{x, y}:= \begin{cases} x_k, & \text{if $k \geq 0$}\\ y_{-k},& \text{if $k<0$}\end{cases}$$ Such a sequence will be called a \textbf{bi-infinite linking sequence from $x$ to $y$ in $(X_0 \cap X_1, J_M(1))$}.
On $X_0 \cap X_1$ let us define the functional $p_{\theta, q}$ given by \begin{equation} p_{\theta, q}(x, y):= \inf \Gamma_{\theta, q}((J_M(2^k;S_k^{x, y}, S_{k+1}^{x, y})),
\end{equation} where the infimum is taken over all approximating sequences of $x$ and $y$.
\begin{definition}[Separator Functional] Let $M \neq \emptyset$ and $p: M \times M \to [0, \infty) $. We say that $p$ is a \textbf{separator functional of points in $M$} and $$p(x, y)=0 \Longleftrightarrow x=y.$$
\end{definition} \begin{theorem} The functional $p_{\theta, q}$ is a separator on $X_0 \cap X_1$.

\end{theorem}
\begin{proof} Let $x, y \in X_0 \cap X_1$ and $(x_n)$ and $(y_n)$ be two arbitrary approximating sequences of $x$ and $y$ in $(X_0 \cap X_1, J_M(1))$ and let $(S_j^{x, y})_{j \in \mathbb{Z}}$ be the bi-infinite linking sequence from $x$ to $y$ in $(X_0 \cap X_1, J_M(1))$. Fixing $N \in \mathbb{N}$, we have $$K_M(1; S_N^{x, y}, S_{-N}^{x, y}) \leq \sum_{j=-N}^{N-1} K_M(1; S_j^{x, y}, S_{j+1}^{x, y}).$$ Thus, taking the $\lim \sup_N$, we have
	$$K_M(1;x, y) \leq \sum_{j \in \mathbb{Z}} K_M(1; S_j^{x, y}, S_{j+1}^{x, y}).$$ Now, using Lemma \ref{lema 12}, with $a=1$ and $b=2^j$ in each term $K_M(1; S_j^{x, y}, S_{j+1}^{x, y})$, we have $$K_M(1;x, y) \leq \sum_{j \in \mathbb{Z}} K_M(1; S_j^{x, y}, S_{j+1}^{x, y}) \leq \sum_{j \in \mathbb{Z}} \min\{1, 1/2^j\} \cdot J_M(2^j; S_j^{x, y}, S_{j+1}^{x, y}).$$ We can also write, for all $\theta \in (0, 1)$: $$K_M(1;x, y) \leq  \sum_{j \in \mathbb{Z}} [2^{j \theta} \cdot \min\{1, 1/2^j\}] \cdot  [2^{-j \theta} \cdot J_M(2^j; S_j^{x, y}, S_{j+1}^{x, y})].$$ Using the \textit{Hölder Inequality} in the last sum, we have \begin{equation}					
		K_M(1; x, y) \leq M_{\theta, q} \cdot \Gamma_{\theta, q}((J_M(2^j; S_j^{x, y}, S_{j+1}^{x, y})_{j \in \mathbb{Z}}), \end{equation}
	where $$M_{\theta, q}:=\biggr[ \sum_{j \in \mathbb{Z}} [2^{j \theta} \cdot \min\{1, 1/2^j\}]^p \biggr]^{1/p}$$ with $\frac{1}{p} + \frac{1}{q}=1$. Note that from this inequality it follows that \begin{equation} \label{inclusao em K(1)}K_M(1; x, y) \leq M_{\theta, q} \cdot p_{\theta, q}(x, y). \end{equation}
	Furthermore, for given $x$ and $y$, the constant approximating sequences $x_n=x$ and $y_n=y$ for all natural $n$ are such that for these, the bi-infinite linking sequence $(S_k^{x, y})$ satisfies: $$S_k^{x, y}=\begin{cases}				
		x, & \text{if $k > 0$}\\ y, & \text{if $k \leq 0$}
	\end{cases}$$ Thus, for this bi-infinite sequence, we have $$\sum_{k \in \mathbb{Z}} [2^{-k \theta} \cdot J_M(2^k; S_k^{x, y}, S_{k+1}^{x, y})]^q = [ J_M(1; x, y)]^q.$$ From this, it follows that \begin{equation} \label{inclusao em J(1)} p_{\theta, q}(x, y) \leq \Gamma_{\theta, q}((J_M(2^k; S_k^{x, y}, S_{k+1}^{x, y})) \leq   J_M(1; x, y).\end{equation} From inequality \ref{inclusao em J(1)}, we have that $p_{\theta, q}(x, y) <+ \infty$ for all $x, y \in X_0 \cap X_1$ and from inequality \ref{inclusao em K(1)} we have that $p_{\theta, q}(x, y)=0 \Longleftrightarrow x=y$. \end{proof}

\begin{lemma} Let $\theta \in (0, 1)$ and $q \in [1, \infty]$. Given $x, y \in X_0 \cap X_1$, let us define the
	functional \begin{equation} P_{\theta, q}(x, y):= \frac{ p_{\theta, q}(x, y)+ p_{\theta, q}(y, x)}{2}.
	\end{equation} The functional $P_{\theta, q}: (X_0 \cap X_1) \times (X_0 \cap X_1) \to [0, \infty)$ satisfies
	
	\begin{enumerate}
		\item $P_{\theta, q}(x, y)=0 \Longleftrightarrow x=y$.
		\item $P_{\theta, q}(x, y)=P_{\theta, q}(y, x)$, for all $x, y \in X_0 \cap X_1$.
		\item $\min\{p_{\theta, q}(x, y); p_{\theta, q}(y, x)\} \leq P_{\theta, q}(x, y) \leq \max\{p_{\theta, q}(x, y); p_{\theta, q}(y, x)\}$, for all $x, y \in X_0 \cap X_1$.				
	\end{enumerate}
\end{lemma}
\begin{proof} The proof is immediate. \end{proof}
\begin{definition} Given $\theta \in (0, 1)$ and $q \in [1, \infty]$ let us define the functional $\delta_{\theta, q}$ on
	$X_0 \cap X_1$,  given by
	\begin{equation} \delta_{\theta, q}(x, y):= \inf \sum_{j=0}^{n-1} P_{\theta, q}(x_j, x_{j+1}),
	\end{equation} where the infimum is taken over all finite sequences of points $x_0=x, x_1, x_2, x_3, \
	\ldots, x_n=y \in X_0 \cap X_1$.
\end{definition}
\begin{theorem} The functional $\delta_{\theta, q}$ as above is a metric on $X_0 \cap X_1$.

\end{theorem}
\begin{proof} It is easy to verify that $\delta_{\theta, q}(x, x)=0$. Furthermore, for all $x, y \in X_0 \cap X_1$, taking the sequence $x_0=x$ and $x_1=y$, we have $$0 \leq \delta_{\theta, q}(x, y) \leq P_{\theta, q}(x, y)< + \infty.$$ Let now $x, y \in X_0 \cap X_1$ such that $\delta_{\theta, q}(x, y)=0$. From inequality \ref{inclusao em K(1)}, we have $$K_M(1; x, y)=K_M(1; y, x) \leq M_{\theta, q} \cdot p_{\theta, q} (x, y),$$ for all $x, y \in X_0 \cap X_1$. Thus, it is also true that $K_M(1; x, y) \leq M_{\theta, q} \cdot p_{\theta, q}(y, x)$. Therefore,
	\begin{equation}K_M(1; x, y) \leq M_{\theta, q} \cdot \min\{p_{\theta, q}(x, y); p_{\theta, q}(y, x)\} \leq M_{\theta, q} \cdot P_{\theta, q}(x, y),\end{equation}  for any $x, y \in X_0 \cap X_1$. For any linking sequence $x_0=x, x_1, x_2, \ldots, x_n=y$ we have \begin{align*}K_M(1; x, y) &\leq \sum_{j=0}^{n-1} K_M(1; x_j, x_{j+1})\\ & \leq M_{\theta, q} \cdot \sum_{j=0}^{n-1} P_{\theta, q}(x_j, x_{j+1}).\end{align*} Consequently, we have \begin{equation} \label{inclusao do interpolado em K(1)}	
		K_M(1; x, y) \leq M_{\theta, q} \cdot \delta_{\theta, q}(x, y), \quad x, y \in X_0 \cap X_1. \end{equation} This inequality guarantees the positivity of $\delta_{\theta, q}$.
	The symmetry of $\delta_{\theta, q}$ comes from the symmetry of $P_{\theta, q}$ and the fact that every linking sequence from $x$ to $y$ in $X_0 \cap X_1$ is also a linking sequence from $y$ to $x$ in $X_0 \cap X_1$.
	For the triangle inequality, let $x, y , z \in X_0 \cap X_1$ and $\varepsilon>0$. Then, there exist linking sequences $x_0=x, x_1, x_2, \ldots, x_n=z$ from $x$ to $z$ and $y_0=z, y_1, y_2, \ldots, y_m=y$ from $z$ to $y$, such that $$\delta_{\theta, q}(x, z) \leq \sum_{j=0}^{n-1} P_{\theta, q}(x_j, x_{j+1}) < \delta_{\theta, q}(x, z)+ \varepsilon$$ and $$\delta_{\theta, q}(z, y) \leq \sum_{j=0}^{m-1} P_{\theta, q}(y_j, j_{j+1}) < \delta_{\theta, q}(z, y)+ \varepsilon.$$ Note that the sequence $$z_0=x, z_1, \ldots, z_n=z, z_{n+1}=z, z_{n+2}, \ldots, z_{n+m}=y$$ is a linking sequence from $x$ to $y$ in $X_0 \cap X_1$ and satisfies \begin{align*}\delta_{\theta, q}(x, y)& \leq \sum_{j=0}^{m+n-1} P_{\theta, q}(z_j, z_{j+1})\\					
		& = \sum_{j=0}^{n-1}P_{\theta, q}(x_j, x_{j+1})+ \sum_{j=0}^{m-1} P_{\theta, q}(y_j, y_{j+1})\\
		&< \delta_{\theta, q}(x, z)+ \delta_{\theta, q}(z, y) + 2 \varepsilon. \end{align*} From the arbitrariness of $\varepsilon>0$ it follows $$\delta_{\theta, q}(x, y) \leq \delta_{\theta, q}(x, z)+ \delta_{\theta, q}(z, y).$$ Therefore, $\delta_{\theta, q}$ is indeed a metric on $X_0 \cap X_1$. \end{proof}
\begin{remark} Inequality \ref{inclusao do interpolado em K(1)} ensures that the metric space $(X_0 \cap X_1, \delta_{\theta, q})$ is continuously included in the metric space $(X_0 \cup X_1, K_M(1))$, thus satisfying the conditions for relative completion. \end{remark}

\begin{definition}[The $R$-interpolated metric space] The \textbf{$R$-interpolated metric space} is denoted by $R_{\theta, q}^{\overrightarrow{X}}$ and is the relative completion of the metric space $$(X_0 \cap X_1, \delta_{\theta, q})$$ in the metric space $(X_0 \cup X_1, K_M(1))$. The metric of this space is denoted by $\Delta_{\theta, q}$ and equals $\Delta_{\theta, q}(x, y) = \delta_{\theta, q}(x, y)$, for all $x, y \in X_0 \cap X_1$.
\end{definition}
\begin{remark} An important fact is that the interpolated spaces above were defined as the completion of the intersection $X_0 \cap X_1$ with a specific metric in the metric space $ (X_0 \cup X_1, K_M(1))$. This ensures that the elements of the interpolated spaces are contained in $X_0 \cup X_1$. However, note that such metrics on the intersection also satisfy the conditions for completion in $(X, d_X)$ which can be substantially ``larger'' than $X_0 \cup X_1$, providing many more new elements to the completions. That is why the name ``relative'' was chosen for this type of completion. We do not study here the advantages of completing the spaces beyond $X_0 \cup X_1$. \end{remark}

\section{Properties of the $R$-Interpolated Metric Space}
\begin{theorem} The inclusions $X_0 \cap X_1 \subseteq R_{\theta, q}^{\overrightarrow{X}} \subseteq X_0 \cup X_1$ are valid and furthermore \begin{enumerate}
		\item $\Delta_{\theta, q}(x, y) = \delta_{\theta, q}(x, y) \leq J_M(1; x, y), \quad x, y \in X_0 \cap X_1$.
		\item $K_M(1; x, y) \leq M_{\theta, q} \cdot \Delta_{\theta, q}(x, y), \quad x, y \in R_{\theta, q}^{\overrightarrow{X}}$.
\end{enumerate} \end{theorem}

\begin{proof} The inclusions are immediate. We will show only the inequalities. Let $x, y \in X_0 \cap X_1$. From inequality \ref{inclusao em J(1)}, we have $$p_{\theta, q}(x, y) \leq J_M(1; x, y).$$ Since $J_M(1; x , y) =J_M(1; y, x)$, it holds that $$\max\{p_{\theta, q}(x, y); p_{\theta, q}(y, x)\} \leq  J_M(1; x, y).$$ Thus, $P_{\theta, q}(x, y) \leq  J_M(1; x, y)$. For the trivial linking sequence $x_0=x$ and $x_1=y$, we have \begin{equation}
		\delta_{\theta, q}(x, y) \leq P_{\theta, q}(x, y) \leq J_M(1; x, y). \end{equation}

	The second inequality is obtained directly from inequality \ref{inclusao do interpolado em K(1)}. Therefore, the inclusions are continuous, as we wanted to demonstrate. \end{proof}
\begin{definition} Let $(X, d_X)$ and $(Y, d_Y)$ be metric spaces and $T:X \to Y$ an operator. We say that
	$T$ is a \textbf{closed operator} if $$[d_X(x_n, x) \to 0 \quad \text{and} \quad d_Y(T(x_n), y) \to 0 ]\implies y=T(x).$$ In other words, the set $$\text{Gr}(T):=\{(x, T(x)) \in X \times Y\}$$ is a closed set in $X \times Y$ in the product topology. \end{definition}

\begin{definition} Let $(X, d_X)$, $(y, d_Y)$ and $T: X \to Y$ an operator. We say that the operator $T$ is
	a \textbf{Lipschitz operator} if there exists a positive constant $C$ such that $$d_Y(T(x), T(y)) \leq C \cdot d_X(x, y), \quad \forall x, y \in X.$$
	When $T$ is a Lipschitz operator, we call the number $$\inf_{C>0}\{d_Y(T(x), T(y)) \leq C \cdot d_X(x, y); \quad \forall x, y \in X\}$$ the \textbf{Lipschitz constant of the operator $T$}. \end{definition}
\begin{lemma}[Reindexing and interpolation factor] \label{reindexacao} Let $(X, d_X)$ and $(Y, d_Y)$ be metric spaces and $T:X \to Y$ an operator. Let \((S_k)_{k\in\mathbb Z}\) be a linking sequence from $x$ to $y$ in $X$ (which for simplicity of notation we will omit the superscript $x, y$) and let
	\[ a_k:=2^{-k\theta}J_M(2^k\lambda;\,T (S_k), T (S_{k+1}): Y),\qquad b_k:=2^{-k\theta}J_M(2^k;\,S_k,S_{k+1}: X), \] where \(\lambda>0\) and \(\theta\in(0,1)\). Suppose that for all \(k\in\mathbb Z\) it holds that \begin{equation}\label{hyp:basic}
		J_M(2^k;S_k,S_{k+1}:X) \;\ge\; \omega_0^{-1}\,J_Y(2^k\lambda;\,T (S_k), T (S_{k+1}):Y), \quad \omega_0>0.
		\tag{H} \end{equation} For $\omega_1>0$ define \(\lambda=\dfrac{\omega_0}{\omega_1}\). Then:
	\begin{enumerate} \item[(i)] If \(\lambda=2^r\) for some \(r\in\mathbb Z\), then, for all \(q\in[1,\
		infty]\), \[ \big\| (a_k)_k\big\|_{\ell^q} \le \omega_0^{\,1-\theta}\omega_1^{\,\theta}\;\big\| (b_k)_k\big\|_{\ell^q}. \]
		\item[(ii)] If \(\lambda>0\) is arbitrary, then, for all \(q\in[1,\infty]\), \[ \big\| (a_k)_k\big\|_{\ell^q} \le C(\theta)\,\omega_0^{\,1-\theta}\omega_1^{\,\theta}\;\big\|(b_k)_k\big\|_{\ell^q}, \] where \(C(\theta)\le 2^{\theta}\) is a constant depending only on \(\theta\).

	\end{enumerate}
	
\end{lemma}
\begin{proof} We start by observing that hypothesis \eqref{hyp:basic} is equivalent, in terms of the terms
	$a_k,b_k$, to \[ b_k \;\ge\; \omega_0^{-1} a_k, \qquad\text{that is}\qquad a_k \le \omega_0\, b_k, \qquad\forall k\in\mathbb Z. \tag{1} \]
	\textbf{If  \(\lambda=2^r\), \(r\in\mathbb Z\):} Writing \(\lambda=2^r\) and making the index change \(m=k+r\) we obtain, for each \(k\), \[ a_k =2^{-k\theta}J_M(2^k\lambda;\,T (S_k),T (S_{k+1}):Y) =2^{-k\theta}J_M(2^{k+r};\,T (S_k),T (S_{k+1}):Y). \] Rewriting in terms of the index \(m\) (with \(m=k+r\)) we have the identity \begin{align*} a_k& =2^{r\theta}\,2^{-(k+r)\theta}J_M(2^{k+r};\,T (S_k),T (S_{k+1}):Y)\\ &=\lambda^{\theta}\,2^{-m\theta}J_M(2^{m};\,T (S_{m-r}),T (S_{m-r+1}):Y). \end{align*} By hypothesis \eqref{hyp:basic} applied to the index \(m-r\) we have \begin{align*} &2^{-m\theta}J_M(2^{m};\,T (S_{m-r}),T( S_{m-r+1}):Y) \leq\\ &\omega_0\,2^{-(m-r)\theta}J_M(2^{m-r};\,S_{m-r},S_{m-r+1}:X). \end{align*} Combining the two expressions and using reindexing (the transformation \(k\mapsto	m=k+r\) is an index shift), we obtain \begin{align*} &2^{-m\theta}J_M(2^{m};\,T (S_{m}),T( S_{m+1}):Y) \le \\ &\omega_0\,\lambda^{-\theta}\,2^{-(m-r)\theta}J_M(2^{m-r};\,S_{m-r},S_{m-
			r+1}:X). \end{align*} Taking the $\ell^q$ norm of both sides and remembering that the index shift is
	isometric in~$\ell^q$ (that is, $\|(u_{m-r})_m\|_{\ell^q}=\|(u_k)_k\|_{\ell^q}$), it results in \begin{align*} &\big\| (2^{-m\theta}J_M(2^{m};T (S_{m}),T (S_{m+1});Y))_m\big\|_{\ell^q} \le\\ & \omega_0\,\lambda^{-\theta}\, \big\|(2^{-k\theta}J_X(2^{k};S_{k},S_{k+1}:X))_k\big\|_{\ell^q}. \end{align*} But the left-hand side is exactly $\|(a_k)_k\|_{\ell^q}$ (after reindexing), and the right-hand side is $\|(b_k)_k\|_{\ell^q}$. Substituting $\lambda=\dfrac{\omega_0}{\omega_1}$ we obtain
	\[ \big\|(a_k)_k\big\|_{\ell^q}
	\le \omega_0\Big(\frac{\omega_0}{\omega_1}\Big)^{-\theta}\big\|(b_k)_k\big\|_{\ell^q} = \omega_0^{\,1-\theta}\omega_1^{\,\theta}\big\|(b_k)_k\big\|_{\ell^q}, \] which proves (i).
	
	\textbf{Arbitrary Case \(\lambda>0\):} Choose \(r=\lfloor\log_2\lambda\rfloor\in\mathbb Z\) such that \(2^r\le\lambda<2^{r+1}\). For all \(k\) we have the two elementary estimates \[ J_M(2^{k+r};\cdot, \cdot:Y)\;\le\; J_M(2^k \lambda;\cdot, \cdot:Y)\;\le\; J_Y(2^{k+r+1};\cdot, \cdot:Y), \] because the function \(s\mapsto J_M(2^k s;\cdot, \cdot:Y)\) is non-decreasing in \(s>0\)
	(remembering that here only the multiplicative parameter within the second argument of \(J_M\) varies). Multiplying by \(2^{-k\theta}\) and reindexing as before, we obtain
	\begin{align*} 2^{r\theta}\,2^{-(k+r)\theta}J_M(2^{k+r};\cdot, \cdot:Y) &\le 2^{-k\theta}J_M(2^k\lambda;\cdot, \cdot:Y)\\ &\le 2^{(r+1)\theta}\,2^{-(k+r+1)\theta}J_M(2^{k+r+1};\cdot, \cdot :Y). \end{align*} Applying inequality \eqref{hyp:basic} (in the form used in (i)) to the appropriate indices	and taking the $\ell^q$ norm, we arrive at inequalities of the type \[ \big\|(a_k)_k\big\|_{\ell^q} \le \omega_0\,2^{\theta}\,\lambda^{-\theta}\big\|(b_k)_k\big\|_{\ell^q} \quad\text{or}\quad \big\|(a_k)_k\big\|_{\ell^q} \le \omega_0\,\lambda^{-\theta}\big\|(b_k)_k\big\|_{\ell^q}, \] depending on whether we use the majorant or the minorant above. In both cases we can
	write compactly \[ \big\|(a_k)_k\big\|_{\ell^q} \le C(\theta)\,\omega_0\,\lambda^{-\theta}\big\|(b_k)_k\big\|_{\ell^q}, \qquad C(\theta)\le 2^{\theta}. \] Substituting \(\lambda=\dfrac{\omega_0}{\omega_1}\) results in \[ \big\|(a_k)_k\big\|_{\ell^q} \le C(\theta)\,\omega_0^{\,1-\theta}\omega_1^{\,\theta}\big\|(b_k)_k\big\|_{\ell^q}, \] as we wanted to demonstrate.	
\end{proof}

\begin{theorem} \label{interpolacao de lip} Let $(X_0, d_{X_0})$ and $(X_1, d_{X_1})$ be a pair of compatible metric spaces with
	respect to the metric space $(X, d_X)$ and $(Y_0, d_{Y_0})$ and $(Y_1, d_{Y_1})$ a pair of compatible metric spaces with respect to the metric space $(Y, d_Y)$. If $(Y, d_Y)$ is complete and $T: X \to Y$ is a closed operator that satisfies: \begin{enumerate}
		\item $T_0:=T|_{X_0}: X_0 \to Y_0$ is Lipschitzian with Lipschitz constant $\omega_0>0$
		\item $T_1:=T|_{X_1}: X_1 \to Y_1$ is Lipschitzian with Lipschitz constant $\omega_1>0$						
	\end{enumerate} Then, $T_{\theta, q}:=T|_{R_{\theta, q}^{\overrightarrow{X}}}: R_{\theta, q}^{\overrightarrow{X}} \to R_{\theta, q}^{\overrightarrow{Y}}$ is Lipschitzian with Lipschitz constant not exceeding $2^{\theta} \cdot  \omega_0^{1- \theta} \cdot \omega_1^{\theta}>0$.
\end{theorem}
\begin{proof} Let $x, y \in X_0 \cap X_1$ and $(x_n)$ and $(y_n)$ be two approximating sequences of $x$ and $y$ in $X_0 \cap X_1$. Let us define the bi-infinite linking sequence $(S_k^{x, y})$ that links $x$ to $y$ in $X_0 \cap X_1$ as usual through these. Since $(x_n)$ is Cauchy in $(X_0 \cap X_1, J_M(1))$, it follows that $(x_n)$ is Cauchy in $(X_0, d_{X_0})$ and in $(X_1, d_{X_1})$, because \begin{align*}J_M(1; x_m, x_n: X)&:= \max\{d_{X_0}(x_m, x_n), d_{X_1} (x_m, x_n)\}\\								&\geq d_{X_i}(x_m, x_n), \quad i=0, 1.\end{align*} Since $(x_n)$ is Cauchy in $(X_i, d_{X_i})$, $i=0, 1$, due to the fact that $$d_{Y_i}(T_i(x_m), T_i(x_n)) \leq \omega_i \cdot d_{X_i}(x_m, x_n), \quad i=0, 1,$$  we have that $(T(x_n))$ is a Cauchy sequence both in $(Y_0, d_{Y_0})$ and in $(Y_1, d_{Y_1})$. Since by hypothesis there exists $C_i>0$ such that $d_{Y}(T(x_m), T(x_n)) \leq C_i \cdot d_{Y_i}(T(x_m), T(x_n)) $,  for $i=0, 1$, we have that $(T(x_n))$ is a Cauchy sequence in $(Y, d_Y)$, which is complete, by hypothesis. It follows that there exists $a \in Y$ such that $d_Y(T(x_n), a) \to 0$. By the closedness property of the operator $T$, we have $a=T(x)$. The same reasoning can be applied to the sequence $(y_n)$ which will provide $b \in Y$ such that $T(y)=b$.		
	Note that if $(T(S_k^{x, y}))$ is a bi-infinite linking sequence from $T(y)$ to $T(x)$ in $Y_0 \cap Y_1$. Thus, since \begin{align*}&J_M(2^k; S_k^{x, y}, S_{k+1}^{x, y}: X)= \\
		&\max\{d_{X_0}(S_k^{x, y}, S_{k+1}^{x, y}); 2^k d_{X_1}(S_k^{x, y}, S_{k+1}^{x, y})\} \geq \\
		& \max\{\omega_0^{-1}  d_{Y_0}(T(S_k^{T(x), T(y)}), T(S_{k+1}^{T(x), T(y)}), 2^k  \omega_1^{-1}  d_{Y_1}(T(S_k^{T(x), T(y)}), T(S_{k+1}^{T(x), T(y)})\}=\\
		&\omega_0^{-1} \cdot J_M(2^k \omega_0/\omega_1; T(S_k^{T(x), T(y)}), T(S_{k+1}^{T(x), T(y)}):Y) \geq \\							
		& 2^{-\theta} \cdot \omega_0^{ \theta-1} \cdot \omega_1^{-\theta} \cdot J_M(2^k; T(S_k^{T(x), T(y)}), T(S_{k+1}^{T(x), T(y)}):Y),
	\end{align*} where in the last inequality we used the previous Lemma \ref{reindexacao}. From this
	inequality, we have that $$ p_{\theta, q; Y}(T(x), T(y)) \leq 2^{\theta} \cdot \omega_0^{1- \theta} \cdot \omega_1^{\theta} \cdot p_{\theta, q; X}(x, y).$$
	From this inequality, it is immediate that \begin{equation} \Delta_{\theta, q; Y}(T(x), T(y)) \leq 2^{\theta} \cdot \omega_0^{1- \theta} \cdot \omega_1^{\theta} \cdot \Delta_{\theta, q; X}(x, y), \quad x, y \in R_{\theta, q}^{\overrightarrow{X}}
	\end{equation} And this concludes the proof.
\end{proof}
\begin{remark} \label{obs sobre a constante lips}
	If we use the result that $$J_M(t; \cdot, \cdot) \leq \max \{1, t/s\} \cdot J_M(s; \cdot, \cdot)$$ we can eliminate the dependence on $\theta$ from the upper bound of the Lipschitz constant of the operator $T$ in the interpolated space. In this situation, setting $t=2^k$ and $s=2^k \omega_0/\omega_1$, we have $$\max\{\omega_0, \omega_1\} \cdot J_M(2^k; S_k, S_{k+1}:X) \geq J_M(2^k; T(S_k), T(S_{k+1}): Y),$$

	resulting in \begin{equation} \label{maximo das constantes}
		\Delta_{\theta, q; Y}(T(x), T(y)) \leq \max\{\omega_0, \omega_1\} \cdot \Delta_{\theta, q; X}(x, y), \quad x, y \in R_{\theta, q}^{\overrightarrow{X}}.
	\end{equation} We choose to display the possible dependence on $\theta$ because it is a constant more
	characteristic of the language in Interpolation Theory and expresses some type of convexity. \end{remark}
\begin{corollary} In Theorem \ref{interpolacao de lip} consider $X=Y$ and $d_X=d_Y$ a metric space complete. If $\omega_i<1$ for $i=0, 1$, and $(X_0, d_{X_0})$ and $(X_1, d_{X_1})$ are both complete then there exists a unique fixed point of $T$ in $X_0 \cup X_1$, even if $T$ is not necessarily Lipschitzian nor continuous in $(X_0 \cup X_1, K_M(1))$. \end{corollary}
\begin{proof} Since $\max\{\omega_0, \omega_1\}<1$, it follows from inequality \ref{maximo das constantes} that $T_{\theta, q}$ (the operator $T$ restricted to the interpolated space) has a unique fixed point $x'$, by the \textit{Banach Fixed Point Theorem}. Such a fixed point necessarily belongs to $X_0 \cup X_1$. Thus, $x' \in X_0$ or $x' \in X_1$. Since by hypothesis the operator $T$ is a strong contraction in $X_0$ and in $X_1$, unique fixed points $x_0 \in X_0$ and $x_1 \in X_1$ exist. Therefore, we must have $x'=x_0=x_1$ and no other distinct point can be fixed by $T$ in $X_0 \cup X_1$.						
\end{proof}
\begin{remark} Throughout the proof we used the same letter $T$ to indicate different operators,
	since these operators were only restrictions of the same $T$ to spaces that were clear in the context through other symbols. \end{remark}
\begin{remark} In Theorem \ref{interpolacao de lip} if we impose the condition of continuity on $T: X \to Y$ we can waive the
	completeness of $(Y, d_Y)$.
\end{remark}
\begin{definition} Let $T: X \to Y$ be an operator between metric spaces. We will say that \textbf{$T$ is a
		compact operator} if for every sequence $(x_n)$ of points in $X$ that is bounded, there exists a subsequence $(x_{n_k})$ of $(x_n)$ such that $(T(x_{n_k}))$ is convergent. \end{definition}

\begin{proposition}\label{metrica F sozinho} If $X_0=X_1 =X$ and $d_0=d_1=d_{X}$. Given $\theta\in(0,1)$ and $q\in[1,\infty]$, we have \[ R_{\theta, q}^{\overrightarrow{X}} = X \] and the metric $\Delta_{\theta, q}$ is equivalent to the metric $d_X$.
\end{proposition}

\begin{proof} For $t>0$ and $u,v\in X$, \[ J_M(t;u,v)=\max\{d_X(u,v),\,t \cdot d_X(u,v)\}= \begin{cases}
		d_X(u,v), & 0<t\le 1,\\ t \cdot d_X(u,v), & t\ge 1.
	\end{cases} \] In particular, for $t=1$ we have $J_M(1;u,v)=d_X(u,v)$, and the metric $K_M(1;\cdot,\cdot)$	
	defined by infimum over linking sequences coincides with the original metric $d_X$ (since the cost per step is always $h_1(\cdot,\cdot)=d_X(\cdot,\cdot)$).
	By inequalities \eqref{inclusao em K(1)} and \eqref{inclusao em J(1)}): \begin{equation}\label{prelimKp}
		K_M(1;u,v) \le M_{\theta,q} \cdot  p_{\theta,q}(u,v), \end{equation} and for constant sequences, \begin{equation}\label{prelimJp}
		p_{\theta,q}(u,v) \le J_M(1;u,v). \end{equation} In the trivial case, $K_M(1;u,v)=d_X(u,v)$ and $J_M(1;u,v)=d_X(u,v)$. Substituting these	equalities into \eqref{prelimKp} and \eqref{prelimJp} we obtain, for all $u,v\in X$, \[ d_X(u,v) \le M_{\theta,q} \cdot p_{\theta,q}(u,v) \qquad\text{and}\qquad p_{\theta,q}(u,v) \le d_X(u,v). \] Rewriting, \begin{equation}\label{sandwich-p}
		\frac{1}{M_{\theta,q}} \cdot d_X(u,v) \;\le\; p_{\theta,q}(u,v) \;\le\; d_X(u,v). \end{equation} The inequalities \eqref{sandwich-p} show that $p_{\theta,q}$ is equivalent to the metric
	$d_X$ (with explicit constants $1/M_{\theta,q}$ and $1$). We have
	\[ P_{\theta,q}(u,v)=\tfrac12\bigl(p_{\theta,q}(u,v)+p_{\theta,q}(v,u)\bigr), \] and \[ \delta_{\theta,q}(x,y)=\inf\sum_{j=0}^{n-1} P_{\theta,q}(x_j,x_{j+1}), \] where the infimum runs over all linking sequences $x_0=x,\dots,x_n=y$ in $X_0\cap X_1=X$. Since $p_{\theta,q}$ satisfies \eqref{sandwich-p}, the same inequalities hold for $P_{\theta,q}$. More precisely, using \eqref{sandwich-p} and the definition of $P_{\theta,q}$, we obtain
	\[ \frac{1}{M_{\theta,q}}\cdot d_X(u,v) \le P_{\theta,q}(u,v) \le d_X(u,v), \qquad \forall u,v\in
	X. \] Now, by the definition of $\delta_{\theta,q}$, we have \[ \frac{1}{M_{\theta,q}}\cdot  d_X(x,y) \le \delta_{\theta,q}(x,y) \le d_X(x,y), \qquad \forall x,y\
	in X. \] Therefore $\delta_{\theta,q}$ is equivalent to $d_X$. Since $R_{\theta, q}^{\overrightarrow{X}}$ was defined as the relative completion of $(X_0\cap X_1,\delta_{\theta,q}) $ in $(X_0\cup X_1,K_M(1))$ and here $X_0\cap X_1=X_0\cup X_1 = X$, it holds that
	\[ R_{\theta, q}^{\overrightarrow{X}} = X. \]
	Since $\Delta_{\theta, q}=\delta_{\theta, q}$ on $X_0 \cap X_1=X$, we have \[ \frac{1}{M_{\theta,q}}\cdot d_X(x,y) \le \Delta_{\theta,q}(x,y) \le d_X(x,y), \qquad \forall x,y\
	\in X \] as we wanted to demonstrate.
\end{proof}
\begin{corollary}\label{compacidade de operadores} Let $(X_0, X_1)_X$ be a pair of metric spaces, $\theta \in (0, 1)$, $q \in [1, \infty]$ and $(F, d_F)$ a complete metric space. If $T$ is an operator that satisfies \begin{itemize}
		\item[(i)] $T: X \rightarrow F$ is closed; \item[(ii)] $T: X_0 \rightarrow F$ is a compact Lipschitzian operator; \item[(iii)] $T: X_1 \rightarrow F$ is Lipschitzian; \item[(iv)] Every sequence bounded in $X_0 \cap X_1$ with respect to the interpolated metric is also bounded with respect to the metric $d_{X_0}$, \end{itemize} then $T_{\theta, q}:=T|_{R_{\theta, q}^{\overrightarrow{X}}}:  R_{\theta, q}^{\overrightarrow{X}} \to F$ is Lipschitzian and compact. \end{corollary}

\begin{proof} The result is immediate from the definitions, proposition \ref{metrica F sozinho} and theorem \ref{interpolacao de lip}.
\end{proof}

\begin{corollary} \label{ponto fixo} Let $(X_0, X_1)_X$ be a pair of compatible metric spaces, $\theta \in (0, 1)$, $q \in [1, \infty)$ and  $T: X \to X$ a closed operator that satisfies \begin{itemize}
		 \item[(a)] $T_0: X_0 \rightarrow X_0$ is a lipschitzian operator with constant $ \omega_0$;
		 \item[(b)] $T_1: X_1 \rightarrow X_1$ is a lipschitzian operator with constant $\omega_1$; 
		 \item[(c)] $X_1$ is closed in $(X_0 \cup X_1, K_M(1))$ and $2^{\theta} \cdot \omega_0^{1- \theta} \cdot \omega_1^{\theta} <1$.
		\end{itemize}
Then $T_1$ has a fixed point.
\end{corollary}

\begin{proof} Due to Theorem \ref{interpolacao de lip}, the operator $T$ restricted to the interpolated space $R_{\theta, q}^{\overrightarrow{X}}$ is lipschitzian with constant not exceeding $2^{ \theta} \cdot \omega_0^{1- \theta} \cdot \omega_1^{\theta}$.  Thus, $T_{\theta, q}$ is a strong contraction. Since $(X, d_X)$ is complete, it follows that the relative completion is also complete. Thus, by \textbf{Banach's Fixed-Point Theorem}, the operator $T_{\theta, q}$ has a unique fixed point $w$ in $R_{\theta, q}^{ \overrightarrow{X}}$. By the density of $X_0 \cap X_1$ in the interpolated space and the continuous inclusion of the interpolated space in the metric space $(X_0 \cup X_1, K_M(1))$, there exists a sequence $(w_n)$ of points in $X_0 \cap X_1$ such that $K_M(1; w_n, w) \to 0$. Since $X_1$ is closed in $(X_0 \cup X_1, K_M(1))$, it follows that $w \in X_1$. Thus, due to the closedness of $T$ in all $X$, it follows that $T_1(w)=w$.	
	
\end{proof}

\section{The $R$-method in Banach Spaces}

Let $(X_0,\|\cdot\|_0)$ and $(X_1,\|\cdot\|_1)$ be Banach spaces continuously embedded 
into a common Hausdorff topological vector space $Z$ over a field $\mathbb{F}$ ($\mathbb{R}$ or $\mathbb{C}$), with $X_0\cap X_1\neq\emptyset$.
 Let us consider the classic $K$-functional due to \textit{Lions-Peetre}[see \cite{LionsPeetre1964}] given by \begin{equation}
	K(t; x):= \inf_{x_0+x_1=x} (\|x_0 \|_0 + t \cdot \|x_1 \|_1), \quad t>0, \quad x \in X_0 + X_1.
\end{equation}
It is known that for all $t>0$ the functional $K(t; \cdot)$ is a norm on $X=X_0 + X_1$ and that $(X_0 +X_1, K(t; \cdot))$ is a Banach space.

Note that $K(1; x) \leq \| x \|_0$ for all $x \in X_0$ and $K(1; y) \leq \| y \|_1$ for all $y \in X_1$. It shows that $(X_0, X_1)_{X}$ is a compatible pair of metric spaces. In this environment, we have

\begin{proposition}
	$R_{\theta, q}^{\overrightarrow{X}}$ is a Banach space.
\end{proposition}

\begin{proof} Note that $0 \in X_0 \cap X_1 \subseteq R_{\theta, q}^{\overrightarrow{X}}$. Now, we define $\| x \|_{\theta, q}:= \Delta_{\theta, q}(x, 0)$ for all $x \in R_{\theta, q}^{\overrightarrow{X}}$. The set $X_0 \cap X_1$ is endowed with a linear structure. By density of $X_0 \cap X_1$ we can define a linear structure in $R_{\theta, q}^{\overrightarrow{X}}$ by putting for all $x, y \in R_{\theta, q}^{\overrightarrow{X}}$ and $\lambda \in \mathbb{F}$ \begin{equation} \label{linearidade} x+\lambda \cdot y:= \lim_{n \to \infty} x_n+\lambda \cdot y_n,\end{equation} where $\|x_n-x \|_{\theta, q} \to 0$ and $\| y_n - y \|_{\theta, q} \to 0$ and $(x_n)$ and $(y_n)$ are sequences of points in $X_0 \cap X_1$ given by the density. By the equality \ref{linearidade}, it follows that $\| \cdot \|_{\theta, q}$ is a norm on $R_{\theta, q}^{\overrightarrow{X}}$ because $\delta_{\theta, q}= \Delta_{\theta, q}$ in $X_0 \cap X_1$.   Furthermore, since $(X_0, \| \cdot \|_0)$ and $(X_1, \| \cdot \|_1)$ are Banach spaces, it is also known that $X=X_0 +X_1$ is a Banach space with the norm $K(1; \cdot)$. Thus, $R_{\theta, q}^{\overrightarrow{X}}$ is a complete metric space, because it is the relative completion of $(X_0 \cap X_1, \| \cdot \|_{\theta, q})$ in $(X_0+X_1, K(1; \cdot))$ and therefore a Banach Space.

\end{proof}

\begin{remark}
	This feature highlights a genuine advantage of the $R$-method. 
	Classical interpolation functors on normed spaces are formulated within a linear
	framework and act exclusively on linear continuous operators. 
	By contrast, the $R$-method is intrinsically metric and does not rely on linear
	structure. When restricted to normed spaces, it still induces an interpolation functor
	and preserves the Lipschitz property of operators with closed graphs, allowing for the
	interpolation of nonlinear mappings. 
	
	In particular, the method enlarges the classical category of admissible operators while
	retaining the functorial behavior expected in interpolation theory, thereby providing a
	natural extension of classical interpolation beyond the linear setting.
\end{remark}

\end{document}